\documentclass[abstracton]{scrartcl} 
\usepackage{authblk}
\usepackage[colorlinks,breaklinks,backref]{hyperref}
\usepackage{hyperref}
\usepackage{backref}
\usepackage{todonotes}
\definecolor{applegreen}{rgb}{0.55, 0.71, 0.0}

\usepackage[top=3cm, bottom=3cm, left=2.5cm, right=2.5cm]{geometry}
\usepackage{amsthm}
\usepackage{amsmath}
\usepackage{enumitem}
\setlist{noitemsep}
\setlist[enumerate,1]{label=(\arabic*)}
\usepackage{graphicx}
\usepackage[english]{babel}
\usepackage{amsfonts}
\usepackage{amscd,color}
\usepackage{pictex}
\usepackage{amssymb}

\usepackage{tikz}
\usetikzlibrary{decorations.pathreplacing}
\usetikzlibrary{calc}
\usetikzlibrary{shapes,snakes}
\usetikzlibrary{arrows}
\tikzstyle{snode}=[circle ,draw=black,fill=white,thick, inner sep=0pt ,minimum size=1.3mm]
\tikzstyle{bnode}=[circle ,draw=black,fill=black,thick, inner sep=0pt ,minimum size=1.3mm]

\DeclareMathOperator*{\Aut}{Aut}

\newtheorem{theorem}{Theorem}
\newtheorem{observation}[theorem]{Observation}
\newtheorem{proposition}[theorem]{Proposition}
\newtheorem{lemma}[theorem]{Lemma}
\newtheorem{corollary}[theorem]{Corollary}

\newcommand*{\claimproof}{Proof of the Claim}

\title{On the distinguishing chromatic number \\ in hereditary graph classes\thanks{An extended abstract announcing the results presented in this paper has been published in the Proceedings of~Eurocomb’25.}}

\author[1,2]{Christoph Brause}
\author[2]{Rafa{\l} Kalinowski}
\author[2]{Monika Pil\'sniak}
\author[1,2]{Ingo Schiermeyer}
\affil[1]{\normalsize TU Bergakademie Freiberg, 09596 Freiberg, Germany}
\affil[2]{\normalsize AGH University of Krakow,  30-059 Krak\'ow, Poland}

\date{\today}

\begin{document}

\maketitle
\begin{abstract}
The distinguishing chromatic number of a graph $G$, denoted $\chi_D(G)$, is the minimum number of colours in a proper vertex colouring of $G$ that is preserved by the identity automorphism only. Collins and Trenk proved that $\chi_D(G)\le 2\Delta(G)$ for any connected graph $G$, and the equality holds for complete balanced bipartite graphs $K_{p,p}$ and for $C_6$.  In this paper, we show that the upper bound on $\chi_D(G)$ can be substantially reduced if we forbid some small graphs as induced subgraphs of $G$, that is, we study the distinguishing chromatic number in some hereditary graph classes.


\smallskip


\end{abstract}


\renewcommand\baselinestretch{1.2}\normalsize

\section{Introduction}

Let $G$ be a graph and $c\colon V(G)\to \mathbb{N}$ be a proper vertex colouring. 
An \emph{automorphism} with respect to the pair $(G,c)$ is a bijective mapping $\varphi\colon V(G)\to V(G)$ such that $c(v)=c(\varphi(v))$ for each $v\in V(G)$, and $vw\in E(G)$ if and only if $\varphi(v)\varphi(w)\in E(G)$ for each $v,w\in V(G)$.
The set of automorphisms with respect to $(G,c)$ is denoted by ${\rm Aut}(G,c)$.
A vertex of $G$ is \emph{fixed} if it is a fixed point of every automorphism of $\Aut(G,c)$.

A proper vertex colouring $c$ of $G$ is \emph{distinguishing} if it fixes every vertex of $G$. The {\it distinguishing chromatic number} of a graph $G$, denoted $\chi_D(G)$, is the minimum number of colours in a proper distinguishing vertex colouring. This concept was introduced by Collins and Trenk~\cite{CT}, who proved the following tight upper bound.
\begin {theorem} {\rm (Collins, Trenk \cite{CT})}
If $G$ is a connected graph with maximum degree $\Delta$, then $$\chi_D(G)\le 2\Delta.$$
with equality if and only if $G\cong K_{\Delta,\Delta}$ or $G\cong C_6$.
\end{theorem}
This concept spawned numerous papers and results. For example, Collins and Trenk showed that $\chi_D(G)= |V(G)|$ if and only if $G$ is a complete multipartite graph. Furthermore, Cavers and Seyffarth~\cite{CS} characterized graphs $G$ with $\chi_D(G)\ge |V(G)|-2$. Laflamme and Seyffarth~\cite{LS} showed that $K_{\Delta,\Delta-1}$ is the only bipartite graph $G$ with $\chi_D(G)=2\Delta(G)-1$. Fijav\v z, Negami, and Sano~\cite{FNS} proved a notably interesting result that $\chi_D(G)\le 5$ for every 3-connected planar graph $G\notin\{K_{2,2,2},C_6+2K_1\}$. Balachandran, Padinhatteer, and Spiga \cite{BPS} found an infinite family of Cayley graphs $G$ with $\chi_D(G)>\chi(G)$ and relatively small automorphism group.

We motivate our paper from a result by Cranston~\cite{Crans}. He confirmed a conjecture of Collins and Trenk by showing that a connected graph $G\not\cong C_6$ of girth at least five satisfies $\chi_D(G)\leq \Delta(G)+1$. This result can be translated into a forbidden induced subgraph setting.

Given a family of graphs $H_1,\ldots,H_k$, a graph $G$ is {\it $(H_1,\ldots,H_k)$-free} if none of the graphs $H_1,\ldots,H_k$ is an induced subgraph of $G$. If $k=1$ and $H_1=H$, we write for ease that $G$ is $H$-free. In terms of forbidden induced subgraphs, the above mentioned result of Cranston reads as follows. 

\begin{theorem}\label{introduction:theorem:Cranston} {\rm (Cranston \cite{Crans})}
If $G$ is a connected $(C_3,C_4)$-free graph, then 
\[\chi_D(G)\leq \Delta(G)+1\] unless $G\cong C_6.$
\end{theorem}

Motivated by the question which of the two induced forbidden subgraphs we may omit for similar findings, we study upper bounds on the distinguishing chromatic number in hereditary graph classes. All our graph classes are defined in terms of forbidden induced graphs from the set $\{C_4, 2K_2,K_{1,3},K_4,K_4-e\}$. Furthermore, all our obtained bounds are tight and we characterize graphs that achieve equality. 

We use standard terminology and notation of graph theory (cf.~\cite{Diestel}). Besides from that, a vertex $v$ of a graph $G$ is called {\it universal} if $N_G[v]=V(G)$. Furthermore, a vertex $v$ is \emph{simplicial} if $N_G(v)$ is a clique. Given a set $S\subseteq V(G)$, we say that a vertex $v$ is {\it complete to} $S$ if $S\subseteq N_G(v)$, and {\it anti-complete to} $S$ if $S\cap N_G(v)=\emptyset$. A set $M\subseteq V(G)$ is a {\it module of $G$} if every vertex outside $M$ is complete or anti-complete to $M$.

If $G$ and $H$ are two vertex-disjoint graphs, then $G+H$ stands for their {\it join}, that is, $G+H$ is the graph with vertex set $V(G)\cup V(H)$ and edge set $E(G)\cup E(H)\cup \{uv:u\in V(G), v\in V(H)\}$.

\section{Dealing with simplicial vertices}

\begin{proposition}\label{simplicial:proposition:connectedness}
If $G$ is a connected graph and the set $S$ of simplicial vertices in $G$ is non-empty, then $G-S$ is connected.
\end{proposition}
\begin{proof}
For the sake of contradiction, we assume that there is a minimal cut-set $S'\subseteq S$ of $G.$ Hence, each vertex of $S'$ has neighbours in at least~$2$ components of $G-S'$. In other words, each vertex of $S'$ has two non-adjacent neighbours, a contradiction to the fact that all vertices of $S'$ are simplicial. Thus, $G-S$ is connected.
\end{proof}

\begin{lemma}\label{simplicial:lemma:reduction}
If $G$ is a connected graph, the set $S$ of simplicial vertices in $G$ is non-empty, and $\chi_D(G)>\chi_D(G-S)$, then either $\chi_D(G)\leq \Delta(G)$
or
$G$ is complete multipartite with a universal vertex.
\end{lemma}
\begin{proof}
Let $G\cong K_{p_1,\ldots,p_r}$, where $r\geq 1$ and $p_1\geq \ldots\geq p_r$, be a complete multipartite graph. Nota that $\chi_D(G)=|V(G)|>\Delta(G)$. As $S\neq \emptyset$, we have $p_i=1$ for $i=2,\ldots,r$. Hence, there are $p_1$ simplicial vertices, and each neighbour of a simplicial vertex is universal.
We continue by assuming that $G$ is not complete multipartite. By Proposition~\ref{simplicial:proposition:connectedness}, $G-S$ is connected.

Let $c'\colon V(G-S)\to [\chi_D(G-S)]$ be a proper distinguishing vertex colouring of $G-S$. 
Note that the relation of two vertices in $S$ having the same neighbours in $V(G)-S$ defines an equivalence relation.
Let $S_1,S_2,\ldots, S_m$ be its classes.
Note that vertices of different classes are non-adjacent as they are simplicial.

We extend $c'$ to a vertex colouring~$c$ by colouring each class $S_i$ as follows. Note that $N_G(S_i)$ is a clique as all the vertices of $S_i$ have the same neighbours and as each vertex of $S_i$ is a simplicial vertex.
Hence, we take the $|S_i|$ smallest colours that are not used by $c'$ on $N_G(S_i)$ and colour the vertices of $S_i$ pairwise differently with these colours.
As vertices of different classes are non-adjacent,~$c$ is a proper vertex colouring of $G$.

We next prove that $c$ is distinguishing. Let $\varphi\in {\rm Aut}(G,c)$ be an arbitrary automorphism, and $v\in V(G-S)$ be an arbitrary vertex.
As $v\notin S$, we also have $\varphi(v) \notin S$. In other words, every automorphism of ${\rm Aut}(G,c)$ maps $V(G-S)$ to $V(G-S)$. As $c'$ fixes $V(G-S)$ in $G-S$, the vertex colouring $c$ fixes all vertices of~$V(G-S)$ in $G$ as well. As the vertices of each~$S_i$ are coloured pairwise differently, $c$ fixes all vertices of $G$ and is a proper distinguishing vertex colouring.

We now count the number of colours used by~$c$. 
Let $k$ be the largest colour on a vertex of~$G$. Clearly, $\chi_D(G)\leq k$.
By our assumption on $G$, we have $k>\chi_D(G-S)$. Hence, there is some~$S_i$ containing a vertex of colour $k$.
However, as $k$ is among the $|S_i|$ smallest colours that are not used by $c'$ on $N_G(S_i)$, we have $k\leq |S_i\cup N_G(S_i)|$.
If in $G-S_i$ there is a vertex $v\in N_G(S_i)$ with a neighbour~$w$ outside $N_G(S_i)$, then $S_i\cup N_G(S_i)\subseteq \{v\}\cup (N_G(v)\setminus\{w\})$ and
\[\chi_G(D)\leq k \leq |S_i\cup N_G(S_i)| \leq  |N_G(v)| \leq \Delta(G),\]
which is our desired result.
Hence, we may assume that in $G-S_i$ every vertex $v\in N_G(S_i)$ has its neighbours in $N_G(S_i)$.
In other words, $V(G)=S_i\cup N_G(S_i)$.
As $S_i$ is complete to the clique $N_G(S_i)$, we have \[\chi_D(G)=\chi_D(G[S_i])+\chi_D(G[N_G(S_i)])=\chi_D(G[S_i])+|N_G(S_i)|.\]
As $G$ is not complete multipartite, there are two non-adjacent vertices, say $s_1,s_2$, in $S_i$ one of which, say~$s_1$, has a private neighbour in~$S_i$.
Now, it is easily seen that a vertex colouring of $G[S_i]$ in which $s_1$ and $s_2$ are coloured alike and all vertices of $S_i\setminus\{s_1\}$ are coloured by pairwise distinct colours is proper and distinguishing in $G[S_i]$, and so $\chi_D(G[S_i])<|S_i|$.
We obtain $\chi_D(G)< |V(G)|=\Delta(G)+1$.
\end{proof}


\section{$C_4$-free graphs and $2K_2$-free graphs}

Let $G$ be a graph, and $c$ be a vertex colouring of $G$. A vertex $u$ is \emph{unique} for $(G,c)$ if 
\begin{itemize}
    \item $c(u)=\Delta(G)+1$,
    \item $\Delta(G)\notin c(N_G(u))$, and
    \item $\Delta(G)\in c(N_G(v))$ for each $v\in V(G)\setminus\{u\}$ with $c(v)=\Delta(G)+1$.
\end{itemize}

\begin{lemma}\label{C4:lemma}
    If $G$ is a connected $C_4$-free graph and $u\in V(G)$ is a vertex such that \[\chi_D(G[N_G[u]])\leq \Delta(G),\] then
    \[\chi_D(G)\leq \Delta(G)+1.\]
\end{lemma}
\begin{proof}
We now prove by induction that there is a proper distinguishing vertex colouring~$c$ of~$G$ on $\Delta(G)+1$ colours for which $u$ is unique for $(G,c)$.

First assume that $u$ is universal. 
Let $c$ be a proper distinguishing vertex colouring of $G[N_G[u]]$ in the colour set $[\Delta(G)+1]\setminus\{\Delta(G)\}$. We may assume, without loss of generality, $c(u)=\Delta(G)+1$.
Note that $u$ is unique for $(G,c)$. Thus, in what follows, we may assume that $u$ is not universal.

Let $M$ be the set of vertices of maximum distance to $u$. Note that $u$ is anti-complete to $M$. Furthermore, $G-M$ is connected, as otherwise there would be a vertex of larger distance to $u$. 
We also note $\Delta(G-M)\leq \Delta(G)$.
By this fact and induction, there is a proper distinguishing vertex colouring 
$c'\colon V(G-M) \to [\Delta(G)+1]$ such that~$u$ is unique for $(G-M,c')$.

Note that the property of two vertices in $M$ having the same neighbours in $V(G)\setminus M$ defines an equivalence relation. We let $M_1,M_2,\ldots, M_p$ be its classes and assume indices such that
\[|N_G(M_i)\setminus M|\leq |N_G(M_k)\setminus M|\quad\textnormal{whenever}\quad i<k.\]
We choose a vertex $m_{k,i}\in N_G(M_k)\setminus (N_G(M_i)\cup M)$ for each pair $i<k$.
We further divide every $M_k$ into two sets $M_k^1$ and $M_k^2$. 
The set $M_k^1$ contains all vertices of $M_k$ that do not have a neighbour in $M_1\cup M_2\cup \ldots \cup M_{k-1}$ while every vertex of $M_k^2$ has a neighbour in $M_1\cup M_2\cup \ldots \cup M_{k-1}$. 

We are now in a position to extend $c'$ to a vertex colouring~$c$ of $G$ by colouring the vertices of~$M$.
For each $M_k^1$, we take the $|M_k^1|$ smallest colours not used on $N_G(M_k^1)\setminus M$ in order to colour the vertices of $M_k^1$ pairwise differently.
Furthermore, we take a vertex ordering $m_1,m_2,\ldots,m_q$ of $M_1^2\cup M_2^2\cup\ldots\cup M_p^2$ and colour a vertex $m_j$ by the smallest colour not used on
$N_G(m_j)\setminus\{m_{j+1},m_{j+2},\ldots,m_q\}$.

As $G$ is $C_4$-free, we have that one of $M_k^1$ and $N_G(M_k^1)\setminus M$ is a clique. 
If $M_k^1$ is a clique, then each of its vertices is universal in $G[M_k^1\cup (N_G(M_k^1)\setminus M)]$, and so every vertex of $M_k^1$ of colour $\Delta(G)$ has a neighbour of colour $\Delta(G)$ in $G$.
If $N_G(M_k^1)\setminus M$ is a clique, then
$c(m)\leq |M_k^1|+|N_G(M_k^1)\setminus M| \leq d_G(v)\leq \Delta(G)$ for each vertex $v\in N_G(M_k^1)\setminus M$ and for each $m\in M_k^1$ as $u$ is not universal.
In other words, if a vertex of $M_k$ receives colour $\Delta(G)+1$, then it has a neighbour of colour $\Delta(G)$.
Thus, $c$ is a proper vertex colouring of $G$ on at most~$\Delta(G)+1$ colours in which $u$ is unique for $(G,c)$.

It remains to prove that $c$ is distinguishing. Let $\varphi\in{\rm Aut}(G,c)$ be arbitrary.
As $u$ is unique, it is fixed. Furthermore, $w$ and $\varphi(w)$ are of the same distance apart from~$u$ for each $w\in V(G)$. Hence, $\varphi(M)=M$ and
$\varphi(V(G)\setminus M)=V(G)\setminus M$. 
In other words, $\varphi$ restricted to $V(G)\setminus M$ is an automorphism of ${\rm Aut}(G-M,c')$. Thus, $c$ fixes all vertices of $V(G)\setminus M$. 

We claim that all vertices of $M_k$ are fixed if all vertices $M_1\cup M_2\cup \ldots \cup M_{k-1}$ are fixed
and prove this claim by contradiction. Hence, for the sake of contradiction, let us suppose that there is an automorphism $\varphi\in{\rm Aut}(G,c)$ and a vertex $m\in M_k$ with $m\neq\varphi(m)$.
As all vertices of $V(G)\setminus M$ are fixed, we see that $m$ and $\varphi(m)$ have the same neighbours in $V(G)\setminus M$. In other words, $\varphi(m)\in M_k$.
We first consider $m\in M_k^2$ or $\varphi(m)\in M_k^2$, and let $m_i\in M_i$ for some $j<k$ be a neighbour of $m$ or $\varphi(m)$.
As $m_i$ is fixed, we get that $m_i$ is a neighbour of both $m$ and $\varphi(m)$. But now $\{m_i,m,\varphi(m),m_{k,i}\}$ induces a $C_4$, a contradiction. Hence,
$m,\varphi(m)\in M_k^1$.
But now $m$ and $\varphi(m)$ are coloured differently by definition, contradicting the fact $\varphi\in{\rm Aut}(G,c)$. 
In other words, our supposition is false and all the vertices of $M_k$ are fixed if all vertices $M_1\cup M_2\cup \ldots \cup M_{k-1}$ are fixed.

Our claim inductively implies that all vertices of $M$ are fixed. Hence, $c$ is distinguishing for~$G$ and our proof is complete.    
\end{proof}

\begin{theorem}\label{C4:theorem}
If $G$ is a connected $C_4$-free graph, then 
\[\chi_D(G)\leq \Delta(G)+2\]
with equality if and only if $G\cong C_6$.
\end{theorem}
\begin{proof}
    If there is some vertex $u\in V(G)$ with $d_G(u)< \Delta(G)$, then $G[N_G[u]]$ can be distinguished by a proper vertex colouring in $\Delta(G)$ colours, and so $\chi_D(G)\leq \Delta(G)+1$ by Lemma~\ref{C4:lemma}.
    If a maximum degree vertex of $G$ is simplicial, then $G$ is complete, and $\chi_D(G)=|V(G)|=\Delta(G)+1$.
    Hence, for the remainder of our proof, we may assume that $G$ is regular and that none of its vertices is simplicial.

    Let us assume that there is some vertex $u\in V(G)$ for which $N_G(u)$ does not induce a complete multipartite graph. In other words, there are three vertices $v_1,v_2,v_3$ in $N_G(u)$ with $v_1v_2,v_2v_3\notin E(G)$ but $v_1v_3\in E(G)$. But now $G[N_G[u]]$ can be distinguished by a vertex colouring in $\Delta(G)$ colours as we can assign $v_1$ and $v_2$ the same and the vertices of $N_G[u]\setminus \{v_1\}$ pairwise different colours. Lemma~\ref{C4:lemma} implies $\chi_D(G)\leq \Delta(G)+1$. Hence, for the remainder of our proof we may assume that $N_G(u)$ induces a complete multipartite graph for each $u\in V(G)$. As $G$ is $C_4$-free, each $N_G(u)$ consists of a clique~$K(u)$ and an independent set~$I(u)$ such that $K(u)$ is complete to $I(u)$. As each $u$ is not simplicial, we have, in particular, $|I(u)|\geq 2$.

    We claim that $G$ is $C_3$-free and we prove this claim by contradiction. Suppose that there is some triangle $T$ in $G$. Let $t$ be an arbitrary vertex of $T$. Furthermore, let $s_1,s_2\in I(t)$ and $t'\in K(t)$.
    As $G$ is regular, $s_1$ has a neighbour $p$ that is not adjacent to $t$. Now $t,p\in I(s_1)$ but $t'\in N_G(s_1)$ is adjacent to~$t$. Hence, $t'$ is also adjacent to $p$.
    Similarly, $s_1,s_2\in I(t')$ and $p$ is adjacent to $s_1$. Hence, $p$ is also adjacent to $s_2$. But now $\{t,s_1,s_2,p\}$ induces a $C_4$, a contradiction. Hence, our supposition is false, $G$ is $C_3$-free, and the desired result follows by Theorem~\ref{introduction:theorem:Cranston}.
\end{proof}

A graph $G$ is \emph{chordal} if it is $C_k$-free for each $k\geq 4$.
A \emph{root}~$r$ in a chordal graph~$G$ is a vertex that has the same distance to all vertices of $V(G)\setminus\{r\}$ with degree at most~$\Delta(G)-1$.

\begin{figure}[h]\label{Ts}
\centering
\begin{tikzpicture}
\draw (-1,1)  node[bnode]{}-- (0,0) node[bnode]{} -- (1,1) node[bnode]{};
\draw (-3,1)  node[bnode]{}-- (0,0) node[bnode]{} -- (3,1) node[bnode]{};

\draw (-3,2)  node[bnode]{}-- (-3,1) node[bnode]{};
\draw (-1,2)  node[bnode]{}-- (-1,1) node[bnode]{};
\draw (3,2)  node[bnode]{}-- (3,1) node[bnode]{};
\draw (1,2)  node[bnode]{}-- (1,1) node[bnode]{};

\draw (-3.8,2)  node[bnode]{}-- (-3,1) node[bnode]{} -- (-2.2,2) node[bnode]{};
\draw (-1.8,2)  node[bnode]{}-- (-1,1) node[bnode]{} -- (-.2,2) node[bnode]{};
\draw (3.8,2)  node[bnode]{}-- (3,1) node[bnode]{} -- (2.2,2) node[bnode]{};
\draw (1.8,2)  node[bnode]{}-- (1,1) node[bnode]{} -- (.2,2) node[bnode]{};
\end{tikzpicture}
\caption{An example of a symmetric tree $T_s$}
\end{figure}
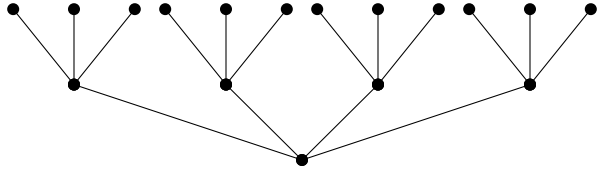

A tree $T_s$ is \emph{symmetric} if all non-leaves have maximum degree, one of these vertices is a root, and every leaf has the same distance to the root (see Fig.\ref{Ts}). 
A graph~$G$ is symmetric if either it is a symmetric tree or it can be constructed from a symmetric tree~$T_s$ by
\begin{itemize}
\item [A)] 
either adding all edges between all leafs of $N_{T_s}(v)$ for each support $v\in V(T_s)$ (such a tree is denoted by $T_A$),
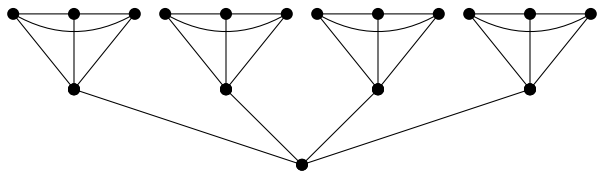
\begin{figure}[h]
\centering
\begin{tikzpicture}
\draw (-1,1)  node[bnode]{}-- (0,0) node[bnode]{} -- (1,1) node[bnode]{};
\draw (-3,1)  node[bnode]{}-- (0,0) node[bnode]{} -- (3,1) node[bnode]{};

\draw (-3,2)  node[bnode]{}-- (-3,1) node[bnode]{};
\draw (-1,2)  node[bnode]{}-- (-1,1) node[bnode]{};
\draw (3,2)  node[bnode]{}-- (3,1) node[bnode]{};
\draw (1,2)  node[bnode]{}-- (1,1) node[bnode]{};

\draw (-3.8,2)  node[bnode]{}-- (-3,1) node[bnode]{} -- (-2.2,2) node[bnode]{};
\draw (-1.8,2)  node[bnode]{}-- (-1,1) node[bnode]{} -- (-.2,2) node[bnode]{};
\draw (3.8,2)  node[bnode]{}-- (3,1) node[bnode]{} -- (2.2,2) node[bnode]{};
\draw (1.8,2)  node[bnode]{}-- (1,1) node[bnode]{} -- (.2,2) node[bnode]{};

\draw (-3.8,2)--(-2.2,2);
\draw[bend right] (-3.8,2) to (-2.2,2) ;
\draw (-1.8,2)--(-.2,2);
\draw[bend right] (-1.8,2) to (-.2,2) ;
\draw (1.8,2)--(.2,2);
\draw[bend right] (.2,2) to (1.8,2) ;
\draw (2.2,2)--(3.8,2);
\draw[bend right] (2.2,2) to (3.8,2);

\end{tikzpicture}
\caption{An example of a symmetric graph $T_A$}
\end{figure}

\item [B)]
or adding all edges between all leafs of $N_{T_s}(v)$ and adding a new vertex $v'$ which is adjacent to all leafs of $N_{T_s}(v)$ for each support $v\in V(T_s)$ (such a tree is denoted by $T_B$).
\end{itemize}
Observe that a complete graph $K_n$ is symmetric, since $K_n=T_A$ for the star $T_s=K_{1,n-1}$.

\begin{figure}[h]
\centering
\begin{tikzpicture}
\draw (-1,1)  node[bnode]{}-- (0,0) node[bnode]{} -- (1,1) node[bnode]{};
\draw (-3,1)  node[bnode]{}-- (0,0) node[bnode]{} -- (3,1) node[bnode]{};

\draw (-3,2)  node[bnode]{}-- (-3,1) node[bnode]{};
\draw (-1,2)  node[bnode]{}-- (-1,1) node[bnode]{};
\draw (3,2)  node[bnode]{}-- (3,1) node[bnode]{};
\draw (1,2)  node[bnode]{}-- (1,1) node[bnode]{};

\draw (-3.8,2)  node[bnode]{}-- (-3,1) node[bnode]{} -- (-2.2,2) node[bnode]{};
\draw (-1.8,2)  node[bnode]{}-- (-1,1) node[bnode]{} -- (-.2,2) node[bnode]{};
\draw (3.8,2)  node[bnode]{}-- (3,1) node[bnode]{} -- (2.2,2) node[bnode]{};
\draw (1.8,2)  node[bnode]{}-- (1,1) node[bnode]{} -- (.2,2) node[bnode]{};

\draw (-3.8,2)--(-2.2,2);
\draw[bend right] (-3.8,2) to (-2.2,2) ;
\draw (-1.8,2)--(-.2,2);
\draw[bend right] (-1.8,2) to (-.2,2) ;
\draw (1.8,2)--(.2,2);
\draw[bend right] (.2,2) to (1.8,2) ;
\draw (2.2,2)--(3.8,2);
\draw[bend right] (2.2,2) to (3.8,2);

\draw (-3,2)  node[bnode]{}-- (-3,2.6) node[bnode]{};
\draw (-1,2)  node[bnode]{}-- (-1,2.6) node[bnode]{};
\draw (3,2)  node[bnode]{}-- (3,2.6) node[bnode]{};
\draw (1,2)  node[bnode]{}-- (1,2.6) node[bnode]{};

\draw (-3.8,2)  node[bnode]{}-- (-3,2.6) node[bnode]{} -- (-2.2,2) node[bnode]{};
\draw (-1.8,2)  node[bnode]{}-- (-1,2.6) node[bnode]{} -- (-.2,2) node[bnode]{};
\draw (3.8,2)  node[bnode]{}-- (3,2.6) node[bnode]{} -- (2.2,2) node[bnode]{};
\draw (1.8,2)  node[bnode]{}-- (1,2.6) node[bnode]{} -- (.2,2) node[bnode]{};

\end{tikzpicture}
\caption{An example of a symmetric graph $T_B$}
\end{figure}
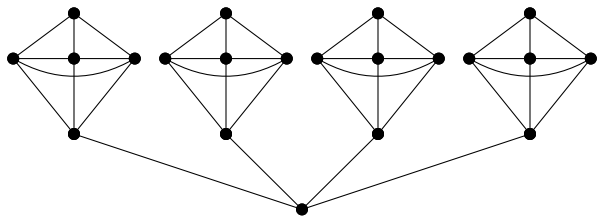

\begin{theorem}\label{simplicial:theorem:chordal}
If $G$ is a connected chordal graph, then \[\chi_D(G)\leq \Delta(G)+1\]
with equality if and only if $G$ is a symmetric graph or $G=\alpha(G)K_1+K_{\omega(G)-1}$. 
\end{theorem}
\begin{proof}
Let $G$ be a chordal graph. It follows from Theorem~\ref{C4:theorem} that $\chi_D(G)\leq \Delta(G)+1$ since $G$ is $C_4$-free and $C_6$ is not chordal. Thus we only have to show that $\chi_D(G)=\Delta(G)+1$ if and only if $G$ is either a symmetric graph or $G=\alpha(G)K_1+K_{\omega(G)-1}$. 

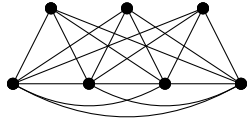
\begin{figure}[h]
\centering
\begin{tikzpicture}
\draw (1,0) node[bnode]{}--(2,0) node[bnode]{}--(3,0) node[bnode]{}--(4,0) node[bnode]{};
\draw[bend right] (1,0) node[bnode]{} to (4,0) node[bnode]{};
\draw[bend right] (1,0) node[bnode]{} to (3,0) node[bnode]{};
\draw[bend right] (2,0) node[bnode]{} to (4,0) node[bnode]{};
\draw (1.5,1) node[bnode]{} -- (1,0);
\draw (2,0)-- (1.5,1) node[bnode]{};
\draw (3,0)-- (1.5,1) node[bnode]{} -- (4,0);
\draw (1,0)-- (2.5,1) node[bnode]{} -- (4,0);
\draw  (3.5,1) node[bnode]{} -- (3,0);
\draw (3,0)-- (2.5,1) node[bnode]{} -- (2,0);
\draw (2,0)-- (3.5,1) node[bnode]{} -- (1,0);
\draw (3.5,1) node[bnode]{} -- (4,0);
\end{tikzpicture}
\caption{Graph $G=3K_1+K_4$}
\end{figure}

Collins and Trenk~\cite{CT} proved that every tree $T$ satisfies the inequality $\chi_D(T)\le \Delta(T)+1$, and the equality holds only for symmetric trees. A distinguishing proper colouring $c$ of a symmetric tree $T_s$ of maximum degree $\Delta$ may be chosen as follows. The central vertex $r$ gets colour $\Delta+1$, and its neighbours get distinct colours from $[\Delta]$.  Then, for every vertex $v$ of $T_s$, its children are coloured with $\Delta-1$ distinct colours from $[\Delta]\setminus \{c(v)\}$. It is easy to see how the colouring $c$ can be extended to symmetric graphs $T_A$ or $T_B$. On the other hand, if $T_A$ or $T_B$ had a proper distinguishing colouring $c'$ with $\Delta$ colours, then it is not difficult to see that the restriction of $c'$ to the corresponding symmetric subtree $T_s$ would also be distinguishing, contrary to the fact that $\chi_D(T_s)=\Delta(T_s)+1$. It is also obvious that for $G=\alpha(G)K_1+K_{\omega(G)-1}$ we have $\chi_D(G)=|G|=\Delta(G)+1$.

Denote $\Delta=\Delta(G)$. Suppose first that $\chi_D(G)=\Delta+1$. The set $S$ of simplicial vertices of $G$ is non-empty since $G$ is chordal (cf. \cite{Ro}). Using induction on the order of a graph, we now prove that $G$ is either a symmetric graph or $G=\alpha(G)K_1+K_{\omega(G)-1}$. Consider the subgraph $G-S$. If $\chi_D(G-S)<\chi_D(G)$, then $G=\alpha(G)K_1+K_{\omega(G)-1}$ by Lemma~\ref{simplicial:lemma:reduction}, and we are done. 

Suppose then that $\chi_D(G-S)=\chi_D(G)$. Hence, $\chi_D(G-S)=\Delta(G-S)+1$, so $\Delta(G-S)=\Delta$ because $\chi_D(G)=\Delta+1$. By the induction hypothesis, the graph $G-S$ is either one of symmetric graph $T_s,T_A,T_B$ or $G-S=\alpha(G-S)K_1+K_{\omega(G-S)-1}$. 

Let $S'$ be the set of simplicial vertices of $G-S$. Clearly, $S'\cap S=\emptyset$. To obtain the graph $G$ from $G-S$, we add vertices of $S$ and some new edges. Since every vertex of $(G-S)-S'$ has degree $\Delta(G-S)=\Delta$, every new edge joins two vertices of $S\cup S'$. Moreover, every vertex of $S'$ has to have an incident new edge, otherwise it would remain simplicial.  

Suppose first $G-S\in \{T_s,T_A,T_B\}$. Let $u_1,\ldots,u_{\Delta}$ be the set of neighbours of root $r$ of $G_S$. For $i=1,\ldots,\Delta$, let $B_{G-S}(u_i)$ denote the {\it branch at} $u_i$ in $G-S$, that is, the connected component of $(G-S)-r$ containing $u_i$. Observe that every proper colouring of the vertices of $B_{G_S}(u_i)$ with $\Delta$ colours is distinguishing and is unique up to a permutation of colours. 

As $G$ is chordal, the branches $B_G(u_i)$ are connected components of $G$ that are pairwise isomorphic. Otherwise, we could colour two neighbours of $r$ with the same colour, and we would not have to use colour $\Delta+1$ to $r$, and to any other vertex, contrary to the assumption that $\chi_D(G)=\Delta+1$. For the same reason, each branch $B_G(u_i)$ has a unique, up to a permutation of colours, proper $\Delta$-colouring. Consequently, every vertex of $S'$ has degree $\Delta$ in $G$, and each component of the subgraph $G[S'\cup S]$ is isomorphic to a graph $D$, which is either a complete graph or a star of suitable order. 

If $G-S$ is a symmetric tree $T_s$, then $G$ is still a symmetric tree when $D=K_{1,\Delta-1}$, or $G$ is $T_A$ when $D=K_{\Delta-1}$, or $G$ is $T_B$ when $D=K_{\Delta}$. If $G-S$ is a symmetric graph $T_A$, then $G$ is $T_B$ when $D=K_{\Delta}$, and there is no other option. Finally, $G-S$ cannot be a symmetric graph $T_B$ because each simplicial vertex in $T_B$ has degree $\Delta-1$. 

Finally, suppose that $G-S=\alpha(G-S)K_1+K_{\omega(G-S)-1}$. Every vertex of the clique $K_{\omega(G-S)-1}$ has degree $\Delta(G)=\Delta$, so it cannot have neighbours in $S$. Every vertex $u$ of $S'$ has degree $\Delta-\alpha(G-S)$, therefore $u$ can be adjacent to at most $\alpha(G-S)<\Delta$ vertices if $S$. It follows that some vertices of $S'$ may have the same colour in a proper distinguishing colouring of $G$. Thus, $G_S$ cannot be a graph $\alpha(G-S)K_1+K_{\omega(G-S)-1}$ when $\chi_D(G-S)<\chi_D(G)$.
\end{proof}

At the end of the Section we consider graphs that are both, $C_4$-free and $2K_2$-free.

\begin{theorem}
If $G$ is a connected $(C_4,2K_2)$-free graph, then \[\chi_D(G)\leq \Delta(G)+1\]
with equality if and only if $G\cong \alpha(G)K_1+K_{\omega(G)-1}$ or $G\cong C_5$. 
\end{theorem}
\begin{proof}
We note that the desired result follows if $G$ is chordal by Theorem~\ref{simplicial:theorem:chordal} or if $G\cong C_5$. Hence, we may assume that $G$ is not chordal and $G\not\cong C_5$.
Thus, there is an induced cycle $C$ of order at least~$4$. As~$G$ is $(C_4,2K_2)$-free, $C$ is of order~$5$.
As $G\not\cong C_5$ but $G$ is connected, $N_G(V(C))\neq \emptyset$.

As $G$ is $(C_4,2K_2)$-free, it follows that $V(C)$ is a module in $G$, $N_G(V(C))$ is a clique, and all vertices which are not in $C$ or $N_G(V(C))$ form in $G$ an independent set, say $I$.
Note that
\[\chi_D(G-I)=\chi_D(C) +\chi_D(G[N_G(V(C))])\leq 3+|N_G(V(C))|\leq d_G(v)-1\leq \Delta(G)-1\]
for each  $v\in N_G(V(C))$. Hence, $I\neq\emptyset$. 
Note that every vertex of $I$ and no vertex of $V(G-I)$ is simplicial in $G$. As further, $G-I$ is not a complete multipartite graph, Lemma~\ref{simplicial:lemma:reduction} implies $\chi_D(G)\leq \Delta(G)$. 
\end{proof}

\begin{lemma}[\cite{ChGTT}]\label{2K2:lemma:dominatingclique}
If $G$ is a connected $2K_2$-free graph, then~$G$ contains a dominating clique of size~$\omega(G)$ unless~$\omega(G)=2$. 
\end{lemma}

\begin{theorem}
If $G$ is a connected $2K_2$-free graph, then
\[\chi_D(G)\leq 2\Delta(G)-\omega(G)+2\]
with equality if and only if $G$ is a complete graph or a balanced complete bipartite graph.
\end{theorem}
\begin{proof}
Note that 
\[\chi_D(G)=|V(G)|= 2\Delta(G)-\omega(G)+2\]
if $G$ is complete or a balanced complete bipartite graph.
Hence, for the remainder of our proof, let us assume that $G$ is distinct from these graphs.
By Cavers and Seyffarth~\cite{CS}, we thus have $\chi_D(G)\leq 2\Delta(G)-1$. In view of the desired result, we may assume $\omega(G)\geq 3$.
Thus, there is a dominating clique~$W$ of size~$\omega(G)$ in~$G$ by Lemma~\ref{2K2:lemma:dominatingclique}.

We take a vertex $w\in W$ which has the highest degree among all vertices of $W$ and colour this vertex with colour $2\Delta(G)-\omega(G)+1$. Furthermore, we colour all neighbours of $w$ by pairwise different colours from~$[d_G(w)]$. As $G$ is distinct from a clique, we have $N_G(w)\setminus W\neq\emptyset$.
As~$W$ is dominating, each vertex of~$G$ which is not in $N_G[w]$ is adjacent to at least one vertex of~$W$. 
We choose a vertex ordering $v_1,v_2,\ldots,v_p$ of $V(G)\setminus N_G[w]$ and colour a vertex $v_i$ by the smallest colour neither used in $N_G(v_i)\setminus\{v_{i+1},v_{i+2},\ldots,v_p\}$ nor on any non-neighbour of $v_i$ in $\{v_1,v_2,\ldots,v_{i-1}\}$ that has the same neighbours in $N_G[w]$ as $v_i$.
Note that we obtain a proper vertex colouring~$c$ of $G$.

Let us assume that there is a vertex $v_i$ with $d_G(v_i)=d_G(w)$ that has colour $2\Delta(G)-\omega(G)+1$ and whose neighbourhood contains a vertex of every colour from $[d_G(w)]$.
Obviously, $v_i$ is non-adjacent to $w$.
Let $V$ be the set of vertices from $\{v_1,v_2,\ldots,v_{i-1}\}$ that have the same neighbours in $N_G[w]$ as $v_i$ and let $w'\in W$ be a neighbour of $v_i$. 
Note that the colour $2\Delta(G)-\omega(G)+1$ of $v_i$ is from $[|N_G(v_i)\cup V|]$.
From
\[2\Delta(G)-\omega(G)+1\leq |N_G(v_i)\cup V|\leq |N_G(v_i)|+|V| \leq d_G(v_i)+(d_G(w')-\omega(G)+1) \leq 2\Delta(G)-\omega(G)+1\]
we obtain that $v_i,w,$ and $w'$ are vertices of maximum degree and $V$ is a set of size $\Delta(G)-\omega(G)+1.$
For a contradiction, let us suppose $N_G(v_i)\neq N_G(w)$. As both neighbourhoods are of the same size, there is some vertex $x\in N_G(w)$ which is not a neighbour of $v_i$. As $N_G(v_i)$ contains a vertex of every colour from $[d_G(w)]$, there is some $y\in N_G(v_i)$ which is of the same colour as~$x$. Note that the vertices of $N_G(w)$ are coloured pairwise differently, and so $y$ is not adjacent to~$w$.
Moreover, $\{v_i,w,x,y\}$ induces a $2K_2$, a contradiction. Hence, we have $N_G(v_i)=N_G(w)$. 
Note that thus any $w\in W\setminus \{w\}$ is adjacent to all vertices of $V$ and all vertices of $W\setminus \{w\}$. But
$|V\cup (W\setminus\{w\})|=\Delta(G).$
Hence, as $W$ is dominating, we have 
\[V(G)=N_G[w]\cup V \quad \textnormal{and}\quad |V(G)|=2\Delta(G)-\omega(G)+2.\]
Thus, it suffices to prove $\chi_D(G)\leq |V(G)|-1$.
Note that $|W\setminus\{w\}|$ is of size at least $2$ and recall that $N_G(w)\setminus W\not=\emptyset$.
We take a vertex $w'\in W$ and a vertex $v\in N_G(w)\setminus W$, and colour the vertices of $G$ such that $v$ and $w'$ receive the same colour and all vertices of $V(G)-w$ receive pairwise distinct colours. Note that this colouring distinguishes $G$, and thus $\chi_D(G)\leq |V(G)|-1$.

Let us now assume that every vertex $v_i$ that receives colour $2\Delta(G)-\omega(G)+1$ has a degree different from $d_G(w)$ or there is a colour on a vertex in $N_G(w)$ that is not used for a vertex of $N_G(v_i)$.
Hence, $c$ fixes $w$. As all vertices of $N_G(w)$ are pairwise differently coloured, $c$ fixes $N_G(w)$ as well.
For the sake of contradiction, let us suppose that there is some $\varphi\in {\rm Aut}(G,c)$ and some $v_i,v_j\in V(G)\setminus N_G[w]$ such that $\varphi(v_i)=v_j$. 
Note that $v_i$ and $v_j$ have the same neighbours in $N_G[w]$ and therefore receive distinct colours by definition, a contradiction to the fact~$\varphi(v_i)=v_j$. 
Hence, our vertex colouring with at most $2\Delta(G)-\omega(G)+2$ colours is distinguishing.
\end{proof}

\section{$claw$-free graphs}

Let $G$ be a connected non-complete graph. 
A non-complete dominating module is \emph{minimal} in $G$ if it cannot be partitioned into two non-complete dominating modules of $G$.
We denote by $p(G)$ the largest integer $p$ such that $V(G)$ can be partitioned into $p$ pairwise disjoint modules $P_1,P_2,\ldots,P_{p}$ of $G$
each of which is non-complete and dominating. Note that each $P_i$ is minimal.
We further let $p(G)=0$ if $G$ is complete.

\begin{lemma}\label{claw:lemma:fixed}
If $G$ is a connected non-complete $claw$-free graph, $P$ is a minimal non-complete dominating module of $G$, $S\subseteq P$ is a set of vertices which induces a non-complete but connected graph, and $c\colon V(G[P])\to \mathbb{N}$ is a vertex colouring that fixes all vertices of $S$, then all vertices of $P$ are fixed.
\end{lemma}
\begin{proof}
Let $S'\subseteq P$ be a maximal set of vertices which induces a connected graph, which contains all vertices of $S$, and whose all vertices are fixed.
We show that $S'$ is dominating and is a module in $G$. By the minimality of $P$ we thus have that $P\setminus S'$ is a non-empty clique whose vertices are coloured pairwise differently by $c$. As $P$ is a dominating module, we also have that no other vertex of $G$ receives a colour from $P\setminus S'$. Hence, $c$ fixes all vertices of $P$, which is our desired result.

We first show that $S'$ is a module by taking an arbitrary vertex $u\in V(G)\setminus S'$ with a neighbour in~$S'$ and showing that $u$ is complete to $S'$. If $u\notin P$, then $u$ is complete to $S'$ as $S'\subseteq P$ and $P$ is a dominating module. Hence, we may assume $u\in P\setminus S'$.
As $u\notin S'$, $u$ is not fixed, i.e.~there is some automorphism $\varphi\in \Aut(G,c)$ with $\varphi(u)\neq u$. Note that $u$ and $\varphi(u)$ are non-adjacent as~$c(u)=c(\varphi(u))$. Consequently, $N_G(u)\cap S' = N_G(\varphi(u))\cap S'$ as all vertices of $S'$ are fixed.
Taking an arbitrary vertex $s_1\in N_G(u)\cap S'$ and $s_2\in N_G(s_1)\cap S'$, we find that $s_2$ is in $N_G(u)\cap S'$ as $\{u,\varphi(u),s_1,s_2\}$ does not induce a $claw$. In other words, since $G[S']$ is connected, $u$ is complete to~$S'$. By the arbitrariness of $u$, we get that $N_G(S')$ is complete to~$S'$, and so $S'$ is a module in $G$.

Finally, we show that $S'$ is dominating. As $S'$ is a non-complete module, there are two non-adjacent vertices $s_1,s_2\in S$ that are complete to $N_G(S')$. As $\{s_1,s_2,u,v\}$ does not induce a $claw$ for each $u\in N_G(S')$ and each $v\in N_G(u)$ and as $G$ is connected, $S'$ is dominating. 
\end{proof}

\begin{figure}[h]
\centering
\begin{tikzpicture}
\draw (0,0) node[bnode]{}--(1,0) node[bnode]{}--(2,0) node[bnode]{}--(3,0) node[bnode]{}--(4,0) node[bnode]{}--(5,0) node[bnode]{};
\draw[bend right] (0,0) node[bnode]{} to (5,0) node[bnode]{};
\draw (0,0)-- (1.5,1) node[bnode]{} -- (1,0);
\draw (3,0)-- (1.5,1) node[bnode]{} -- (4,0);
\draw (0,0)-- (2.5,1) node[bnode]{} -- (5,0);
\draw (3,0)-- (2.5,1) node[bnode]{} -- (2,0);
\draw (2,0)-- (3.5,1) node[bnode]{} -- (1,0);
\draw (5,0)-- (3.5,1) node[bnode]{} -- (4,0);
\end{tikzpicture}
\caption{The line graph $L(K_{1,3})$}
\end{figure}
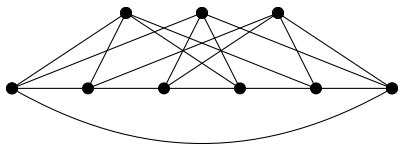

\begin{observation}
If $G\cong L(K_{1,3})$, then $\chi(G)=3$ and $\chi_D(G)=5$.
\end{observation}

\begin{theorem}\label{claw:theorem:main}
If $G$ is a connected $claw$-free graph, then \[\chi_D(G)\leq \chi(G)+p(G)\] unless $G\cong C_6$ or $G\cong L(K_{1,3})$.
\end{theorem}
\begin{proof}
For a complete graph, every proper vertex colouring is distinguishing. Hence, in view of a proof for the desired result, we may assume that $G$ is non-complete, and so $p(G)>0$.
We further assume that $G$ is distinct from a $6$-cycle.

Let $P_1,P_2,\ldots,P_{p(G)}$ be a partition of $V(G)$ into $p(G)$ pairwise disjoint non-complete dominating modules of $G$. 
Note that each $P_i$ is indeed a minimal.
We let $G_i$ be the graph $G[P_i]$.
Note that
\[\chi_D(G)=\sum_{i=1}^{p(G)} \chi_D(G_i)\quad\textnormal{and}\quad\chi(G)=\sum_{i=1}^{p(G)} \chi(G_i).\]
It suffices to show \[\chi_D(G_i)\leq \chi(G_i)+1\]
for each $i$.

Let $i\in [p(G)]$ be arbitrary and let $c\colon V(G_i)\to [\chi(G_i)]$ be a proper vertex colouring of $G_i$ with colour classes $X_1,X_2,\ldots,X_{\chi(G_i)}$.

By Lemma~\ref{claw:lemma:fixed}, we obtain $\chi_D(G_i)\leq \chi(G_i)+1$ if there is a set $I\subseteq [\chi(G_i)]$ and a component~$G'$ of $G[\bigcup_{i\in I} X_i]$ such that $G'$ is not complete and there exists a proper distinguishing vertex colouring~$c'$ of $G'$ in colours $I\cup\{\chi(G_i)+1\}$.
This fact can be seen as follows: First of all, we may assume without loss of generality that at least one vertex of $G'$ gets colour $\chi(G_i)+1$ by $c'$. We now define a new vertex colouring~$c_i$ of $G_i$ on $\chi(G_i)+1$ colours as follows:
$$c_i(v)=\begin{cases}
 c'(v) &\textnormal{if } v\in V(G'),\\
 c(v) &\textnormal{if } v\notin V(G').\\
\end{cases}
$$
Note that every automorphism $\varphi\in{\rm Aut}(G_i,c_i)$ maps $V(G')$ to $V(G')$ as this set contains the only vertices in $G_i$ of colour~$\chi(G)+1$ and induces a component of $G[\bigcup_{i\in I} X_i]$. As $c'$ is distinguishing for $G'$, we have that $c_i$ fixes all vertices of $V(G')$ in $G_i$. Hence, by Lemma~\ref{claw:lemma:fixed} completes our proof as  $V(G_i)$ is a minimal non-complete dominating module of $G$ and as $G'$ is non-complete but connected.

Hence, for every set $I\subseteq [\chi(G_i)]$ and every component~$G'$ of $G[\bigcup_{i\in I} X_i]$, we have that $G'$ is complete or every proper distinguishing vertex colouring of $G'$ uses at least $|I|+2$ colours.
As every bipartite $claw$-free graph is a path or a cycle and these graphs have distinguishing chromatic number at most~$3$ if they are distinct from a $4$-cycle or $6$-cycle, we find that every component of $G[X_i\cup X_j]$ is a $2$-path, $4$-cycle, or $6$-cycle for each distinct~$i,j\in[\chi(G)]$.

First of all, we assume that we have a bichromatic $6$-cycle $C$ in $G_i$, i.e.~there are distinct $i,j\in[\chi(G)]$ such that $C$ is a component of $G[X_i\cup X_j]$.
As $G$ is $claw$-free, we have $p(G)=1$, and so $G=G_i$. As $G$ is distinct from a $6$-cycle but connected, there is a colour class $X_k$ such that $k\neq i,j$ and $C$ is not a component of $G[X_i\cup X_j\cup X_k]$. In particular, there exist vertices in $S:=X_k\cap N_G(V(C))$. 
As a first subcase, let us assume that there is a vertex $s\in S$ such that $N_G(s)\cap V(C)$ does not induce $2K_2$.
As $G$ is $claw$-free, $s$ is not complete to $V(C)$. In other words, $s$ has a neighbour, say $u$, and a non-neighbour on $C$. As $C$ is connected, we may assume that $u^+$ is a not adjacent to $s$. As $\{u^-,u,u^+,s\}$ does not induce a $claw$, $u^-$ is adjacent to $s$.
Note that $c(u^-)=c(u^+)$ and $s$ has a neighbour of colour $c(u^+)$ while $u^+$ does not. 
Consequently, by recolouring $s$ and $u^+$ by colour $\chi(G_i)+1$, we obtain a vertex colouring of $G_i$ in which $s$ and $u^+$ are fixed.
For a contradiction, suppose that there is a vertex apart from $u$ in $V(C)$ which has neighbours $u^+$ and $s$. Clearly, this vertex has to be $u^{2+}$.
But now $\{u^+,u^{2+},u^{3+},s\}$ induces a $claw$ if $u^{3+}s\notin E(G)$, $N_G(x)\cap V(C)$ induces a $2K_2$ if $u^{3+}s\in E(G)$ but $u^{4+}s\notin E(G)$, and $\{u,u^{2+},u^{4+},s\}$ induces a $claw$ if $u^{3+}s,u^{4+}s\in E(G)$, a contradiction. 
Hence, $u$ is the only vertex of $C$ that is adjacent to $u^+$ and $s$, and so $u$ is fixed as well. Now Lemma~\ref{claw:lemma:fixed} with $P=V(G_i)$ implies that our colouring is distinguishing as $\{s,u,u^+\}$ induces a non-complete but connected graph.
As our second subcase, let us assume that $N_G(s)\cap V(C)$ induces a $2K_2$ for each $s\in S$.
We fix one $u\in V(C)$ and define
\begin{itemize}
\item $S_u:=\{s: N_G(s)\cap V(C)=\{u,u^+,u^{3+},u^{4+}\}\}$,
\item $S_{u^+}:=\{s: N_G(s)\cap V(C)=\{u^+,u^{2+},u^{4+},u^{5+}\}\}$, and 
\item $S_{u^{2+}}:=\{s: N_G(s)\cap V(C)=\{u,u^{2+},u^{3+},u^{5+}\}\}$.
\end{itemize}
Note that $S_u,S_{u^+},S_{u^2+}$ is a partition of $N_G(V(C))$ and we may assume, without loss of generality $|S_u|\geq |S_{u^+}|, |S_{u^{2+}}|$.
As $G$ is $claw$-free, each of these three sets is a clique, they are pairwise anti-complete, and each of their neighbours is a vertex on $C$ or in $N_G(V(C))$.
In particular, \[V(G)=V(C)\cup S_1\cup S_2\cup S_3.\]
Note that $\chi(G)=2+|S_u|.$
We claim $|S_u|\geq 2$ by supposing the contrary for a contradiction. 
As $G$ is distinct from $L(K_{1,3})$, we get $|S|\leq 2$.
We fix some $s\in S$ and choose $s'\in S\setminus\{s\}$ if $S\setminus\{s\}\neq \emptyset$.
Hence, there is a vertex $u'\in V(C)$ that is adjacent to $s$ but non-adjacent to $s'$ if it is chosen.
By recolouring $u'$, we obtain a distinguishing vertex colouring of $G[V(C)\cup \{s,s'\}]$. By our assumption on components of $G[X_i\cup X_j\cup X_j]$, we have that $G[V(C)\cup \{s,s'\}]$ is not a component of $G[X_i\cup X_j\cup X_k]$, a contradiction.
Hence, $|S_u|\geq 2$ as claimed.
We now choose a vertex $s\in S_u\setminus X_k$ and recolour $s$ as well as $u^{2+}$ by colour $\chi(G)+1$.
Note that $\{u,u^+,s\}$ is the only clique of size~$3$ in $G$ whose vertices are coloured by $c(X_i),c(X_j)$, and $\chi(G)+1$.
Hence, all its vertices are fixed. As an immediate consequence, all vertices of $S_u$ are fixed as $S_u\cup\{u,u^+\}$ is the only clique of size $|S_u|+2$ with $s\in S_u$. Furthermore, as $s$ and $u^{2+}$ are the only vertices of colour $\chi(G)+1$ but $s$ is fixed, $u^{2+}$ is fixed as well.
Now, our obtained colouring in $\chi(G)+1$ colours is distinguishing by Lemma~\ref{claw:lemma:fixed} as $V(G_i)$ is a minimal non-complete dominating module and $\{u, u^+,u^{2+}\}$ induces a non-complete but connected graph.

For the sake of a contradiction, we suppose next that we have a bichromatic $4$-cycle $C$ in $G_i$, i.e.~there are distinct $i,j\in[\chi(G)]$ such that $C$ is a component of $G[X_i\cup X_j]$. Let $u\in V(C)$.
We first claim that $G_i$ is connected by supposing the contrary for a contradiction. As $G$ is connected, there is a vertex $v\in V(G)\setminus V(G_i)$. But now $\{u,u^{2+},v,w\}$ induces a $claw$ for each $w\in V(G_i)$ that is in a different component as the vertices of $C$, a contradiction. Thus, $G_i$ is connected as claimed.
By the minimality of the dominating module $V(G_i)$, the set $\{u,u^{2+}\}$ is not a dominating module as otherwise $\{u,u^{2+}\}$ and $V(G_i)\setminus \{u,u^{2+}\}$ would be two dominating modules of $G_i$.
As $G$ is $claw$-free, $\{u,u^+\}$ is not a module if its is not dominating. Hence, there has to be a vertex $v\in N_G(V(C))$ with a neighbour and a non-neighbour on $V(C)$. 
Note that $v\in V(G_i)$ and $v\in X_k$ for some $k\neq i,j$.
Let~$G'$ be the component of $G[X_i\cup X_j\cup X_k]$ that contains the vertices of $V(C)$. Note that $\chi_D(G')\geq 5$ by assumption.
As $G$ is $claw$-free, $v$ has two adjacent or three neighbour on~$C$. Furthermore, from the existence of $v$, it follows that each two vertices of $X_{k}$ do not have the same neighbourhood on~$C$. 
We claim that every vertex all neighbours of $V(C)$ in $X_k$ have $3$ or $4$ neighbours on $C$ by supposing the contrary for a contradiction.
Let $v\in X_k$ be a vertex with two neighbours on~$C$. Clearly, $N_G(v)\cap V(C)$ is a clique. Recolouring $v$ and one of its non-neighbour on $C$ by colour $\chi(G_i)+1$ leads to a proper distinguishing vertex colouring of $G'$ on four colours, a contradiction.
Hence, all neighbours of $V(C)$ in $X_k$ have $3$ or $4$ neighbours on $C$ as claimed. 
As $G$ is $claw$-free, this means \[V(G')=V(C)\cup [X_k\cap N_G(V(C))].\]
As every vertex of $X_k\cap N_G(V(C))$ has at least three neighbours on $C$ and as $C$ is $claw$-free, we have $|X_k\cap N_G(V(C))|\leq 2$.
We recall that $v$ has three neighbours on $C$ and it is non-adjacent to either $u$ or $u^{2+}$, say $u$.
As recolouring $u^+$ by colour~$\chi(G)+1$ does not lead to a distinguishing vertex colouring on four colours of $G'$, it follows that $u$ has a neighbour in $X_k$ that is non-adjacent to $u^{2+}$ but adjacent to $u^-$ and $u^+$.
But now, it is easily seen that recolouring $u$ and $v$ by colour $\chi(G)+1$ leads to a proper distinguishing vertex colouring of $G'$ on $4$ colours, a contradiction.

Finally, we have to consider the case in which every vertex of $G_i$ has at most one neighbour of each colour in $c$, i.e.~$G[X_i\cup X_j]$ has maximum degree at most~$1$ for each distinct $i,j\in[\chi(G)]$. Recolouring an arbitrary vertex $v$ with colour $\chi(G_i)+1$ leads to a vertex colouring on $\chi(G_i)+1$ colours of $G_i$.
Clearly, $v$ is fixed. Furthermore, if all vertices of distance $i-1$ to $v$ are fixed, then, as each of them has at most one neighbour of a particular colour, all vertices of distance~$i$ to $v$ are fixed as well. By induction on the distance to $v$, it is now easily seen that all vertices are fixed.
\end{proof}

\begin{theorem}\label{claw:corollary:aim}
If $G$ is a connected $claw$-free graph of order~$n$, then \[\chi_D(G)\leq \Delta(G)+2\] with equality if and only if $G\cong C_6$ or ${G}\cong K_{n/2}[2K_1]$.
\end{theorem}
\begin{proof}
We first establish the corollary for $p(G)\leq 1$. Theorem~\ref{claw:theorem:main}, Brook's theorem, and the facts $\chi(H)=3$ and $\Delta(H)=4$ imply
\[
	\chi_D(G)\leq 
	\left.
		\begin{cases}
			\chi(G)+ 2& \textnormal{if } G\cong C_6 \textnormal{ or } G\cong L(K_{1,3}),\\
			\chi(G)+ p(G)& \textnormal{otherwise}
		\end{cases}	
		\right\}
	 \leq 
		\begin{cases}
			\Delta(G)+ 2& \textnormal{if } G \textnormal{ is an odd cycle or } G\cong C_6,\\
			\Delta(G)+ 1& \textnormal{otherwise}.
		\end{cases}
\]
However, it is easily seen $\chi_D(C_{2k+1})=3$ for each $k$ and $\chi_D(C_6)=4$, which proves the corollary for~$p(G)\leq 1$.

We continue by considering $p(G)\geq 2$. Let $P_1,P_2,\ldots,P_{p(G)}$ be a partition of $V(G)$ into $p(G)$ pairwise disjoint non-complete dominating modules of $G$. 
Note that each $P_i$ is minimal.
We let $G_i$ be the graph $G[P_i]$. Since $G$ is $claw$-free, we have that $G_i$ is $3K_1$-free for each $i$. 
If $G_i$ is connected for some $i$, then $\chi_D(G_i)\leq \Delta(G_i)+1$ as the corollary holds for $p(G)=1$, and so
\[\chi_D(G)=\chi_D(G_i)+\chi_D(G-V(G_i))\leq \Delta(G_i)+1+|V(G)\setminus V(G_i)| \leq \Delta(G)+1.\]
Hence, we may assume that each $G_i$ is disconnected for each $i$. As $G_i$ is $3K_1$-free, it is a disjoint union of two cliques.
However, it is easily seen that every proper vertex colouring of $K_n\cup K_m$ is distinguishing unless $n=m$. But for $n=m$, we have $\chi_D(K_n\cup K_n)=n+1$.
If $G_i$ has maximum degree at least~$1$ for some $i$, then $\chi_D(G_i)\leq |V(G_i)|-1$, and so, for $j\neq i$,
\[\chi_D(G)=\chi_D(G_j)+\chi_D(G-V(G_j))\leq \Delta(G_j)+2+\chi_D(G_i)+|V(G)\setminus V(G_i\cup G_j)| \leq \Delta(G)+1.\]
Hence, we may assume that every $G_i$ has maximum degree~$0$. We obtain $G\cong K_{n/2}[2K_1]$. It is easily seen $\chi(K_{n/2}[2K_1])=|V(K_{n/2}[2K_1])|=\Delta(K_{n/2}[2K_1])+2$, which completes our proof.
\end{proof}

\section{($claw,diamond$)-free graphs}

In this section, we will use the concept of {\it proper distinguishing edge-colourings} introduced by Kalinowski and Pil\'sniak in ~\cite{KP}. The {\it distinguishing chromatic index} of a graph $G$, denoted by $\chi'_D(G)$,  is  the minimum number of colours in a proper edge-colouring $c$ such that each vertex of $G$ is a fixed point of every automorphism of $G$ that preserves the colouring $c$. This invariant is defined for all graphs without a component $K_2$. 

\begin{theorem} {\rm \cite{KP}} \label{distindex}
Every connected graph $H$ of order $n\ge 3$ satisfies the inequality 
$$\chi'_D(H)\le \Delta(H)+1$$
except for four graphs $C_4, K_4, C_6, K_{3,3}$. 
\end{theorem}

The equality obviously holds for Class 2 graphs. Moreover, $\chi'_D(C_n)=\chi_D(C_n)$ since $L(C_n)\cong C_n$ for $n\ge 3$, that is, $\chi'_D(C_n)=4$ for $n\in\{4,6\}$, and $\chi'_D(C_n)=3$ otherwise. We also have $\chi'_D(K_4)=\chi'_D(K_{3,3})=5$ as shown in  \cite{KPPW}. 

A graph $K_4-e$ is called a {\it diamond}.

\begin{theorem}
If $G$ is a connected ($claw,\,diamond$)-free graph, then $$\chi_D(G)\le \Delta(G)+1$$  unless $G\in\{C_4,C_6\}$.
\end{theorem}
\begin{proof}
Every ($claw$,$diamond$)-free graph $G$ is a line graph $L(H)$ of some graph $H$. This is because every Beineke graph (cf. \cite{Bein}) different from the claw contains a diamond as an induced subgraph. Let $G=L(H)$. By the well-known Whitney isomorphism theorem, the canonical homomorphism from $\Aut(H)$ to $\Aut(L(H))$ is an isomorphism for every connected graph $H$, except for three graphs $K_2$, $K_{1,3}+e$ (a triangle with a pendant edge) and a diamond $K_4-e$ (cf. \cite{Ju}). However, both graphs $L(K_{1,3}+e)$ and $L(K_4-e)$ contain an induced diamond, contrary to the assumption.

Let then $G=L(H)$ with $H\notin\{K_2,K_{1,3}+e,K_4-e\}$. Whitney isomorphism theorem easily implies that $\chi_D(L(H))=\chi'_D(H)$. 
By Theorem~\ref{distindex}, $\chi_D(G)\le \Delta(H)+1$. It is easy to see that $\Delta(G)\ge \Delta(H)-1$ and the equality holds if and only if every edge incident to a vertex of degree $\Delta(H)$ is a pendant vertex in $H$. It follows that $H$ is a star $K_{1,n-1}$, hence $G=L(K_{1,n-1})$ is a complete graph $K_{n-1}$ with $\chi_D(K_{n-1})=\Delta(K_{n-1})+1$. To complete the proof, it suffices to note that $\chi_D(L(H))=\Delta(L(H))+1$ for $H\in\{K_4,K_{3,3}\}$, since then $\Delta(L(H))=4$ and $\chi_D(L(H))=\chi'_D(H)=5.$
\end{proof}

The difference between $\Delta(H)$ and $\Delta(L(H))$ can be arbitrarily large. 
Observe that if $G$ is a line graph of a regular graph $H$, then $\chi_D(G)\le \frac 12\Delta(G)+2$. Indeed, if $H$ is a $\Delta$-regular graph, then each vertex of $G=L(H)$ has degree $2(\Delta-1)$, and therefore $$\chi_D(G)\leq \chi'_D(H)\le \Delta +1=\frac{\Delta(G)}{2}+2.$$

\begin{theorem}
For every connected ($claw,\;diamond,\;K_k$)-free graph with $k\ge 3$ and $\Delta(G)\ge 3$,  $$\chi_D(G)\leq k.$$
\end{theorem} 
\begin{proof}
Let $G$ be a ($claw$,$diamond$,$K_k$)-free graph. Then $G=L(H)$ for some graph $H$ with $\Delta(H)\le k-1$ since $G$ is $K_k$-free. 
 By Theorem~\ref{distindex}, $\chi'_D(H)\le k$ for every graph $H$. 
\end{proof}

\begin{corollary}
If a graph $G$ is ($claw,\;diamond,\; K_4$)-free, then $\chi_D(G)\leq 4.$   
\end{corollary}


\begin{thebibliography}{99}

\bibitem{BPS} N.\,Balachandran, S.\,Padinhatteer, and P.\,Spiga, Vertex transitive graphs $G$ with $\chi_D(G)>\chi(G)$ and small automorphism group. \newblock \emph{Ars Math. Contemp.}, 17(1): (2019) 311--318.

\bibitem{Bein} L.\,W.\,Beineke,  Characterizations of Derived Graphs. \newblock \emph{ J. Combin. Theory}, 9:(1970) 129--135.

\bibitem{CS} M.\,Cavers  and  K.\,Seyffarth,   Graphs  with  large  distinguishing  chromatic number. \newblock \emph{Electron. J. Combin.}, 20(1): (2013) \#P19.

\bibitem{ChGTT} F.R.K.\,Chung, A.\, Gy\'arf\'as, Zs. Tuza, and W.T.\,Trotter, The maximum number of edges in $2K_2$ graphs with bounded maximum degree, \newblock \emph{Discrete Math.}, 81 (1990) 129--135.

\bibitem{CT}  K.\,Collins and A.\, Trenk,  The distinguishing chromatic number. \newblock \emph{Electron. J. Combin.}, 13(1): (2006) \#R16.

\bibitem{Crans} D.\,Cranston, Proper Distinguishing Colorings with Few Colors for Graphs with Girth at Least~5. \newblock \emph{Electron. J. Combin.} 25(3): (2018) \#P3.5.

\bibitem{Diestel} R.\,Diestel, Graph Theory, 5th Edition. Springer-Verlag, Berlin, Heidelberg, New York 2016.

\bibitem{FNS} G. Fijav\v z, S. Negami, and T. Sano,  3-connected planar graphs are 5-distinguishing colorable with two exceptions. \newblock \emph{Ars Math. Contemp.}, 4 (2011) 165--175.

\bibitem{Ju} H. A. Jung, Zu einem Isomorphiesatz von H. Whitney f\"ur Graphen. \newblock \emph{Math. Ann.}, 164 (3): (1966) 270--271.

\bibitem{KP} R.\,Kalinowski and M.\,Pil\'sniak, Distinguishing graphs by edge-colourings. \newblock \emph{European J. Combin.}, 45: (2015) 124--131.

\bibitem{KPPW} R.~Kalinowski, M.~Pil\'sniak, J.\,Przyby{\l}o, and M.\,Wo\'{z}niak, How to personalize the vertices of a~graph? \newblock \emph{European J. Combin.}, 40 (2014) 116--123.

\bibitem{LS} C.\,Laflamme  and  K.\,Seyffarth,   Distinguishing  chromatic  numbers  of  bipartite  graphs. \newblock \emph{Electron.  J.  Combin.},  16(1): (2009) \#R76.

\bibitem{Ro} D.J. Rose, Triangulated graphs and the elimination process. \newblock \emph{J. Math. Anal. Appl.}, 32 (3): (1970) 597--609.





\end{thebibliography}
\end{document}